\documentclass[11pt]{amsart}

\usepackage[pdftex]{graphicx}
\usepackage[a4paper,margin=2.5cm]{geometry}

\usepackage{amsfonts}
\usepackage{amsmath}
\usepackage{amssymb}
\usepackage{amsthm}
\usepackage{dsfont}
\usepackage{tikz}

\usepackage{rotating}

\usepackage[T1]{fontenc}

\usepackage{paralist}
\usepackage{color}

\usepackage{overpic}

\usepackage[colorlinks,cite color=blue,pagebackref=true,pdftex]{hyperref}

\renewcommand\emptyset{\varnothing}
\renewcommand\P{P}
\renewcommand\k{\mathbf{k}}

\newcommand\bad{\mathcal{B}}
\newcommand\ord{\mathcal{O}}
\newcommand\ordC{\mathcal{L}}
\newcommand\T{\mathcal{T}}
\newcommand\I{I_\bullet}
\newcommand\precC{\prec\mathrel\cdot}
\renewcommand\l{\lambda}
\newcommand\Z{\mathbb{Z}}

\newcommand\R{\mathbb{R}}               
\newcommand\nR{\R_{\ge0}}
\newcommand\Q{\mathcal{Q}}
\newcommand\Qess{\Q_{\mathrm{ess}}}

\newcommand\GT{\mathrm{GT}}
\newcommand\OGT{\mathcal{GT}}

\newcommand\Ehr{\mathsf{Ehr}}
\renewcommand\int{\mathsf{int}}

\newcommand\set[2]{\left\{ {#1} \ : \ {#2} \right\}}
\newcommand\tset[2]{\{ {#1} \ : \ {#2} \}}

\newcommand\Emph[1]{\textbf{#1}}

\DeclareMathOperator*\conv{conv}
\DeclareMathOperator\Ext{Ext}
\DeclareMathOperator*{\codim}{codim}

\DeclareMathOperator*{\relint}{relint}
\DeclareMathOperator*{\dd}{dd}

\newtheorem{thm}{Theorem}[section]
\newtheorem{cor}[thm]{Corollary}
\newtheorem{lem}[thm]{Lemma}
\newtheorem{prop}[thm]{Proposition}

\newtheorem{quest}{Question}

\theoremstyle{definition}

\newtheorem{example}[thm]{Example}

\newtheorem{obs}{Observation}

\title[Arithmetic of marked order polytopes]{Arithmetic of marked order
polytopes, monotone triangle reciprocity, and partial colorings}

\author{Katharina Jochemko}
\author{Raman Sanyal}
\address{Fachbereich Mathematik und Informatik, %
Freie Universit\"at Berlin, %
Germany}
\email{\{jochemko,sanyal\}@math.fu-berlin.de}


\date{\today}
\thanks{KJ was supported by a \emph{Hilda Geiringer Scholarship} at
the Berlin Mathematical School. 
RS has been supported by
the European Research Council under the European Union's Seventh Framework
Programme (FP7/2007-2013) / ERC grant agreement n$^\mathrm{o}$ 247029.}

\begin{document}

\begin{abstract}
    For a pair of posets $A \subseteq P$ and an order preserving map $\l : A
    \rightarrow \R$, the marked order polytope parametrizes the order
    preserving extensions of $\l$ to $P$. We show that the function counting
    integral-valued extensions is a piecewise polynomial in $\l$ and we prove
    a reciprocity statement in terms of order-reversing maps.  We apply our
    results to give a geometric proof of a combinatorial reciprocity for
    monotone triangles due to Fischer and Riegler (2011) and we consider the
    enumerative problem of counting extensions of partial graph colorings
    of Herzberg and Murty (2007).
\end{abstract}

\maketitle

\section{Introduction}

Partially ordered sets, or \Emph{posets} for short, are among the most
fundamental objects in combinatorics. For a finite poset $P$,
Stanley~\cite{stanley70} considered the problem of counting (strictly) order
preserving maps from $P$ into $n$-chains and showed that many problems in
combinatorics can be cast into this form. Here, a map $\l : P \rightarrow [n]$
into the $n$-chain is \Emph{order preserving} if $\l(p) \le \l(q)$ whenever $p
\prec_P q$ and the inequality is strict for \Emph{strict} order preservation.
In \cite{stanley70} it is shown that the number of order preserving maps into
a chain of length $n$ is given by a polynomial $\Omega_P(n)$ in the positive
integer $n$ and the number of strictly order preserving maps is related to
$\Omega_P(n)$ by a combinatorial reciprocity (see
Section~\ref{ssec:reciprocity}). 

In this paper we consider the problem of counting the number of order
preserving extensions of a map $\l : A \rightarrow \Z$ from a subposet $A
\subseteq P$ to $P$. Clearly, this number is finite only when $A$ comprises
all minimal and maximal elements of $P$ and we tacitly assume this throughout.
It is also obvious that no extension exists unless $\l$ is order preserving
for $A$ and we define $\Omega_{P,A}(\l)$ as the number of order preserving
maps $\widehat{\l} : P \rightarrow \Z$ such that $\widehat{\l}|_A = \l$. By
adjoining a minimum and maximum to $P$ it is seen that $\Omega_{P,A}(\l)$
generalizes the order polynomial.

The function $\Omega_P(n)$ can be studied from a geometric perspective by
relating it to the Ehrhart function of the \Emph{order
polytope}~\cite{stanley86}, the set of order preserving maps $P \rightarrow
[0,1]$.  The finiteness of $P$ asserts that this is indeed a convex polytope
in the finite dimensional real vector space $\R^P$. The order polytope is a
lattice polytope whose facial structure is intimately related to the structure
of $P$ and which has a canonical unimodular triangulation again described in
terms of the combinatorics of $P$. Standard facts from Ehrhart theory (see for
example~\cite{BR}) then assert that $\Omega_P(n)$ is a polynomial of degree
$|P|$. We pursue this geometric route and study the \Emph{marked order
polytope}
\[
    \ord_{P,A}(\l) \ = \ \set{ \ \widehat{\l}: P \rightarrow \R \text{ order
    preserving } }{
    \widehat{\l}(a) = \l(a) \text{ for all } a \in A \ } \ \subset \ \R^P\,.
\]
Marked order polytopes were considered (and named) by Ardila, Bliem, and
Salazar~\cite{ABS} in connection with representation theory. In the case that
$A$ is a chain, the polytopes already appear in~\cite{stanley81}; see
Section~\ref{ssec:chains}.  The set $\ord_{P,A}(\l)$ defines a polyhedron for
any choice of $A \subseteq P$ but it is a polytope precisely when $\min(P)
\cup \max(P) \subseteq A$. It follows that $\Omega_{P,A}(\l) =
\#(\ord_{P,A}(\l) \cap \Z^P)$. In Section~\ref{sec:mop} we elaborate on the
geometric-combinatorial properties of $\ord_{P,A}(\l)$ and we show that
$\Omega_{P,A}(\l)$ is a piecewise polynomial over the space of integral-valued
order preserving maps $\l : A \rightarrow \Z$. We give an explicit description
of the polyhedral domains for which $\Omega_{P,A}(\l)$ is a polynomial and we
give a combinatorial reciprocity for $\Omega_{P,A}(-\l)$. We close by
``transferring'' our results to the marked chain polytopes of~\cite{ABS}.

In Section~\ref{sec:MT}, we use our results to give a geometric
interpretation of a combinatorial reciprocity for \Emph{monotone triangles}
that was recently described by Fischer and Riegler~\cite{FR}. A monotone
triangle is a triangular array of numbers such as
\begin{equation}\label{eqn:MTexample}
\begin{array}{ccccccccc}
          & &   & & 5 & &    & &  \\ 
          & &   &4&   &5&    & &  \\ 
          & & 3 & & 5 & & 7  & &  \\
          &3&   &4&   &7&    &8&  \\
        1 & & 4 & & 6 & & 8 & & 9
    \end{array}
\end{equation}
with fixed bottom row such that the entries along the directions $\searrow$
and $\nearrow$ are weakly increasing and strictly increasing in direction
$\rightarrow$; a more formal treatment is deferred to Section~\ref{sec:MT}.
Monotone triangles arose initially in connection with alternating sign
matrices~\cite{MRR} and a significant amount of work regarding their
enumerative behavior was done in~\cite{fischer06}. In particular, it was shown
that the number of monotone triangles is a polynomial in the strictly
increasing bottom row. In~\cite{FR} a (signed) interpretation is given for the
evaluation of this polynomial at weakly decreasing arguments in terms of
\Emph{decreasing monotone triangles}. In our language, monotone triangles are
extensions of order preserving maps over posets know as \Emph{Gelfand-Tsetlin
patterns} plus some extra conditions.  These extra conditions can be
interpreted as excluding the lattice points in a natural subcomplex of the
boundary of $\ord_{P,A}(\l)$. We investigate the combinatorics of this
subcomplex and give a geometric interpretation for the combinatorial
reciprocity of monotone triangles.

Finally, a well-known result of Stanley~\cite{stanley73} gives a combinatorial
interpretation for the evaluation of the chromatic polynomial $\chi_G(t)$
of a graph $G$ at negative integers in terms of acyclic orientations. We give
a combinatorial reciprocity for the situation of counting extensions of
partial colorings which was considered by Herzberg and Murty~\cite{herzberg}. 

\textbf{Acknowledgments.} This project grew out of the course ``Combinatorial
Reciprocity Theorems'' given by the second author in the winter term 2011/12
at FU Berlin. We would like to thank the participants Katharina M\"olter and
Tran Manh Tuan as well as Thomas Bliem for helpful discussions.

\section{Marked order polytopes}\label{sec:mop}

Marked order polytopes as defined in the introduction naturally arise as
sections of a polyhedral cone, the \Emph{order cone}, which parametrizes order
preserving maps from a finite poset $P$ to $\R$. The order cone is the
``cone-analog'' of the order polytope which was thoroughly studied
in~\cite{stanley86} and whose main geometric results we reproduce before
turning to marked order polytopes.  We freely make use of concepts from
polyhedral geometry as can be found, for example, in~\cite{ziegler}. For a
finite set $S$ we identify $\R^S$ with the vector space of real-valued
functions $S \rightarrow \R$.

\subsection{Order cones}
The order cone is the set $\ordC(P)
\subseteq \R^P$ of order preserving maps from $P$ into $\R$
\[
    \ordC(P) \ = \ \{ \phi \in \R^P \;:\; \phi(p)  \ \le \ \phi(q) \text{ for
    all } p \preceq_P q\}.
\]
This is a closed convex cone and the finiteness of $P$ ensures that $\ordC(P)$
is polyhedral (i.e.\ bounded by finitely many halfspaces). The cone is not
pointed and the lineality space of $\ordC(P)$ is spanned by the indicator
functions of the connected components of $P$. Said differently, the largest
linear subspace contained in $\ordC(P)$ is spanned by the functions  $\chi : P
\rightarrow \{0,1\}$ that satisfy $\chi(p) = \chi(q)$ whenever there is a
sequence $p = p_0p_1\dots p_{k-1}p_k = q$ such that $p_ip_{i+1}$ are
comparable in $P$.

The cone $\ordC(P) \subseteq \R^P$ is of full dimension $|P|$ and its facet
defining equations are given by $\phi(p) = \phi(q)$ for every cover
relation $p \precC_P q$. Every face $F \subseteq \ordC(P)$ gives rise to a
subposet $G(F)$ of $P$ whose Hasse diagram is given by those $p \precC_P q$
for which  $\phi(p) = \phi(q)$  for all $ \phi \in F$.  Such a subposet $G(F)$
arising from a face $F \subseteq \ordC(P)$ is called a \Emph{face partition}.
The following characterization of face partitions is taken
from~\cite{stanley86}.
\begin{prop} \label{prop:face_graph}
    A subposet $G \subseteq P$ is a face partition if and only if for every $p,q
    \in G$ with $p \preceq_G q$ we have $[p,q]_P \subseteq G$.
\end{prop}
Equivalently, the directed graph obtained from the Hasse diagram of $P$ by
contracting the cover relations in $G$ is an acyclic graph and, after removing
transitive edges, is the Hasse diagram of a poset that we denote by $P/G$.
Note that $G$ is typically not a connected poset.  The face corresponding to
such a graph $G$ is then
\[
    F_P(G) \ = \ \set{ \phi \in \ordC(P)}{ \phi \text{ is constant on every
    connected component of $G$ } }
\]
and $F_P(G)$ is isomorphic to $\ordC(P/G)$ by a linear and lattice preserving
map.

The order cone has a canonical subdivision into unimodular cones that stems
from refinements of $P$ induced by elements of $\ordC(P)$. To describe the
constituents of the subdivision, recall that $I \subseteq P$ is an \Emph{order
ideal} if $p \preceq_P q$ and $q \in I$ implies $p \in I$. Let $\phi \in
\ordC(P)$ be an order preserving map with range $\phi(P) = \{ t_0 < t_1 <
\cdots < t_k \}$. Then $\phi$ induces a chain of order ideals 
\[
    \I^P \ : \ 
    I_0 \ \subsetneqq \ 
    I_1 \ \subsetneqq \ 
    I_2 \ \subsetneqq \ 
    \cdots \ \subsetneqq \ 
    I_k \ = \ P 
\]
by setting $I_j = \set{ p \in P }{\phi(p) \le t_j}$. If the poset $P$ is clear
from the context, we drop the superscript and simply write $\I$.
Conversely, a given chain
of order ideals $\I$ is induced by $\phi \in \ordC(P)$ if and only if $\phi$
is constant on $I_{j} \setminus I_{j-1}$ for $j=0,1,\dots,k$ (with $I_{-1} =
\emptyset$) and
\[
    \phi(I_0) \ < \
    \phi(I_1\setminus I_0) \ < \
    \phi(I_2\setminus I_1) \ < \
    \cdots  \ < \
    \phi(I_k\setminus I_{k-1}).
\]
This defines the relative interior of a $(k+1)$-dimensional simplicial cone in
$\ordC(P)$ whose closure we denote by $F(\I)$. Chains of order ideals are
ordered by refinement and the maximal elements correspond to saturated chains
of order ideals or, equivalently, \Emph{linear extensions} of $P$. For a
saturated chain $\I$, we have $I_{j} \setminus I_{j-1} = \{ p_j \}$ for
$j=0,1,\dots,m = |P|-1$  and $p_i \prec_P p_j$ implies $i < j$. In this case
\[
    F(\I) \ = \ \set{ \phi \in \R^P}{ 
    \phi(p_0) \ \le \
    \phi(p_1) \ \le \
    \cdots \ \le \
    \phi(p_{m-1}) }.
\]
Modulo lineality space, this is a unimodular simplicial cone spanned by the
characteristic functions $\phi^0,\phi^1,\dots,\phi^{m-1} : P \rightarrow \{0,1\}$ with
$\phi^k(p_j) = 1$ iff $j \ge k$. Faces of $F(\I)$ correspond to the
coarsenings of $\I$ and since every $\phi \in \ordC(P)$ induces a unique $\I =
\I(\phi)$, this proves the following result which was first shown by
Stanley~\cite{stanley86} for the order polytope $\ordC(P) \cap [0,1]^P$.

\begin{prop}
    Let $P$ be a finite poset.  Then 
    \[
    \T_\P \ = \  \set{ F(\I^\P) }{ \I^\P \text{ chain of order ideals in $\P$
        } } 
    \]
    is a subdivision of $\ordC(P)$ into unimodular simplicial cones.
\end{prop}

\subsection{Marked order polytopes} \label{ssec:mop}
Now let $A \subseteq P$ be a subposet such that $\min(P) \cup \max(P)
\subseteq A$. For an order preserving map $\l : A \rightarrow \R$, the
marked order polytope 
\[
    \ord_{P,A}(\l) \ = \ \set{ \widehat{\l} \in \ordC(P) }{ \widehat{\l}(a) =
    \l(a) \text{ for all } a \in A} \ = \ \ordC(P) \ \cap \ \Ext_{P,A}(\l)
\]
is the intersection of the order cone with the affine space $\Ext_{P,A}(\l) =
\tset{ \widehat{\l} \in \R^P }{\widehat{\l}|_A = \l }$. Every face of
$\ord_{P,A}(\l)$ is a section of a face $H$ of $\ordC(P)$ with
$\Ext_{P,A}(\l)$ and is itself a marked order polytope. We denote the
dependence of $H$ on $\l$ by $H(\l)$.
We can describe them in terms of face partitions.
\begin{prop}\label{prop:compat_face_graph}
    Let $G$ be a face partition of $P$ and let $\l : A \rightarrow \R$ be an
    order preserving map for $A \subseteq P$. Then  $\Ext_{P,A}(\l)$ meets
    $F_P(G)$ in the relative interior if and only if the following holds for
    all $a,b \in A$: Let
    $G_a,G_b \subseteq P$ be the connected components of $G$ containing $a$ and $b$,
    respectively.
    \begin{enumerate}[\normalfont i)]
        \item If $\l(a) < \l(b)$ then
            \[
                \bigcup_{p \in G_a} P_{\preceq p} \ \cap \
                \bigcup_{q \in G_b} P_{\succeq q}  \ = \ \emptyset.
                \]
        \item If $\l(a) = \l(b)$ and $a$ and $b$ are comparable, then $G_a =
            G_b.$ 
    \end{enumerate}
    In this case, $F_P(G) \cap \Ext_{P,A}(\l)$ is linearly isomorphic to
    $\ord_{P/G,A/G}(\l_G)$ where $\l_G : A/G \rightarrow \R$ is the
    well-defined map on the quotient.
\end{prop}
\begin{proof}
    Let $P/G$ be the quotient poset associated to the face partition $G$.  The
    quotient $A/G$ is a subposet of $P/G$ and $\l_G : A/G \rightarrow \R$ is a
    well-defined map if condition i) holds. Moreover, the induced map $\l_G$
    is order preserving for $A/G$ if condition i) holds, and in fact strictly
    if ii) holds. Thus $F_P(G) \cap \Ext_{P,A}(\l)$ is linearly isomorphic to
    $\ord_{P/G,A/G}(\l_G)$ which is of maximal dimension.
\end{proof}

We call a face partition \Emph{compatible} with $\l$ if it satisfies the
conditions above. In particular, taking the intersection of all compatible
face partitions of $P$, we obtain $\ord_{P,A}(\l)$ as an improper face.

\begin{cor} \label{cor:dim}
    Let $A \subseteq P$ be a pair of posets and $\l : A \rightarrow \R$ an
    order preserving map. Then $\ord_{P,A}(\l)$ is a
    convex polytope of dimension 
    \[
        \dim \ord_{P,A}(\l) \ = \ \left|
         P \setminus \set{ p \in P }{ a \preceq p \preceq b \text{ for
        } a,b \in A \text{ with } \l(a) = \l(b) } \right|.
    \]
\end{cor}
\begin{proof}
    The presentation  as the affine section of a cone marks $\ord_{P,A}(\l)$
    as a convex polyhedron. As every element of $P$ has by assumption a lower
    and upper bound in $A$, it follows that $\ord_{P,A}(\l)$ is a polytope. 
    The right-hand side is exactly the number of elements of $P$ whose values
    are not yet determined by $\l$ and $\ord_{P,A}(\l)$ has at most this
    dimension. On the other hand, Lemma~\ref{lem:prod_simpl} shows the
    existence of a subpolytope of exactly this dimension.
\end{proof}

\subsection{Induced subdivisions and arithmetic}

Intersecting every cell of the canonical subdivision $\T_P$ of $\ordC(P)$ with
the affine space $\Ext_{P,A}(\l)$ induces a subdivision of $\ord_{P,A}(\l)$
that we can explicitly describe. To describe the cells in the intersection,
let $\I$ be a chain of order ideals of $P$. For $a \in P$ we denote by
$i(\I,a)$ the smallest index $j$ for which $a \in I_j$.  We call a chain of
order ideals $\I$ of $P$ \Emph{compatible} with $\l$ if
\[
    i(\I,a) < i(\I,b) \quad\text{ if and only if }\quad \l(a) < \l(b)
\]
for all $a,b \in A$. The crucial observation is that $\relint F(\I)\cap\Ext_{P,A}(\l)$ is not empty iff $\I$ is compatible with $\l$ and in this case $F(\I) \cap
\Ext_{P,A}(\l)$ is of a particularly nice form.

\begin{lem}\label{lem:prod_simpl}
    Let $A \subseteq P$ be a pair of posets and $\l : A \rightarrow \R$ an
    order preserving map. If $\I$ is a chain of order ideals of $P$ compatible
    with $\l$, then the induced cell $F(\I) \cap \Ext_{P,A}(\l)$ is a
    Cartesian product of simplices.
\end{lem}
\begin{proof}
    Let $\l(A) = \{ t_0  <  t_1 < \dots < t_r \}$ be the range of $\l$ and
    pick elements $a_0,a_1,\dots, a_r \in A$ with $\l(a_i) = t_i$. 
    Let $i_j = i(\I,a_j)$ for $j = 0,1,\dots,r$ and,  
    since $\I$ is compatible with $\l$, we have $0 = i_0 < i_1 < \cdots < i_r =
    k$. It follows that 
    $F(\I) \cap \Ext_{P,A}(\l)$ is the set of all $\phi \in \R^P$ such that
    $\phi$ is constant on $I_h \setminus I_{h-1}$ for $h=0,1,\dots,k$ (with
    $I_{-1}=\emptyset$) and
    \[
    \newcommand{\verteq}[0]{\begin{turn}{90} $=$\end{turn}}
    \begin{array}{ccccccccccccccc}
        \phi(I_0) & \le &
        \phi(I_1 \setminus I_0) & \le &
        \cdots & \le &
        \phi(I_{i_1} \setminus I_{i_{1}-1}) & \le & 
        \phi(I_{i_1+1} \setminus I_{i_{1}}) & \le & 
        \cdots & \le & 
        \phi(I_k \setminus I_{k-1})\\[5pt]
        \verteq && && && \verteq && && && \verteq\\
        \l(a_0) && && && \l(a_1) && && && \l(a_r)\\
    \end{array}
    \]
    Thus, $F(\I)\cap \Ext_{P,A}(\l)$ is linearly isomorphic to 
    $F_0 \times F_1 \times \cdots \times F_{r-1}$ where, by setting \mbox{$s_j =
    \phi(I_j \setminus I_{j-1})$,}
    \begin{equation}\label{eqn:ord_simplex}
    F_j \ = \ \{ \ 
    \l(a_{j}) \ \le \
    s_{i_{j}+1} \ \le \ 
    s_{i_{j}+2} \ \le \ 
    \cdots  \ \le \ 
    s_{i_{j+1}-1} \ \le \ \l(a_{j+1}) \ \}.
\end{equation}
is a simplex of dimension $d_j = i_{j+1} - i_j-1$.
\end{proof}

Thus the canonical subdivision of $\ordC(P)$ induces a subdivision of
$\ord_{P,A}(\l)$ into products of simplices indexed by compatible chains of
order ideals. This is the key observation for the following result. 

\begin{thm}\label{thm:piecew_poly}
    Let $A \subseteq P$ be a pair of posets with $\min(P) \cup \max(P)
    \subseteq A$. For integral-valued order preserving maps $\l : A \rightarrow \Z$, the function
    \[
        \Omega_{P,A}(\l) \ = \ |\ord_{P,A}(\l) \cap \Z^P|
    \]
    is a piecewise polynomial over the order cone $\ordC(A)$. The cells of the
    canonical subdivision of $\ordC(A)$ refine the domains of polynomiality of
    $\Omega_{P,A}(\l)$. In other words, $\Omega_{P,A}(\l)$ is a polynomial
    restricted to any cell $F(\I^A)$ of the subdivision of $\ordC(A)$.
\end{thm}
\begin{proof}
    Lemma~\ref{lem:prod_simpl} shows that for fixed $\l : A
    \rightarrow \Z$ every maximal cell in
    the induced subdivision of $\ord_{P,A}(\l)$ is a product of simplices and
    the proof actually shows that, after taking successive differences, the
    simplices $F_j$ of \eqref{eqn:ord_simplex} are lattice isomorphic to 
    \begin{equation}\label{eqn:std_simplex}
        (\l(a_{j+1}) - \l(a_{j}))\cdot 
        \Delta_{d_j} \ = \ \bigl\{ y \in \nR^{d_j} \;:\; y_1 \ + \ y_2 \ + \
        \cdots \ + \ y_{d_j} \ \le \ 
        \l(a_{j+1}) - \l(a_{j})
         \bigr\}.
    \end{equation}
    Elementary counting then shows that
    \begin{equation}\label{eqn:prod_ehr}
|F(\I)\cap\Ext_{P,A}(\l) \cap \Z^P| \quad = \quad
    \prod_{j=0}^{r-1} |F_j \cap \Z^P| \quad =  \quad
    \prod_{j=0}^{r-1}
    \binom{\l(a_{j+1}) -
    \l(a_j) + d_j}{d_j}
    \end{equation}
    which is a polynomial in
    $\l$ of  degree $d_0+d_1+\cdots+d_{r-1} = \dim F(\I)\cap\Ext_{P,A}(\l)$.
    M\"obius inversion on the face lattice of the induced
    subdivision shows that $\Omega_{P,A}$ is the evaluation of a polynomial at
    the given $\l$. 
    To complete the proof, note that $\l, \l^\prime : A \rightarrow \R$ have
    the same collections of compatible chains of order ideals whenever $\l,
    \l^\prime \in \relint C$ for some cell $C$ in the canonical subdivision  $\T_A$ of
    $\ordC(A)$.
\end{proof}

A weaker version of Theorem~\ref{thm:piecew_poly} can also be derived from the
theory of partition functions~\cite[Ch.~13]{DCP}. It can be seen that over
$\ordC(A)$, the marked order polytope is of the form
\[
    \ord_{P,A}(\l) \ = \ \set{ x \in \R^n }{ Bx \ \le \ c(\l) }
\]
where $B \in \Z^{M \times n}$ is a fixed matrix with $n = |P|$ and $c : \R^A
\rightarrow \R^M$ is an affine map. Moreover, $B$ is unimodular. It follows
from the theory of partition functions that the function $\Phi_B : \Z^M
\rightarrow \Z$ given by
\[
    g \ \mapsto \ \#\{ x \in \Z^n : Bx \le g \}
\]
is a piecewise polynomial over the cone $C_B \subset \R^M$ of (real-valued)
$g$ such that the polytope above is non-empty. The domains of polynomiality
are given by the \Emph{type cones} for $B$; see McMullen~\cite{mcmullen73}.
Consequently, we have $\Omega_{P,A}(\l) = \Phi_B(c(\l))$. It follows that
$\ordC(A)$ is linearly isomorphic to a section of $C_B$ and the canonical
subdivision $\T_A$ is a refinement of the induced subdivision by type cones.
It is generally difficult to give an explicit description of the subdivision
of $C_B$ by type, not to mention the sections of type cones by the image of
$c(\l)$. So, an additional benefit of the proof presented here is the explicit
description of the domains of polynomiality.

In the context of representation theory, the lattice points of certain marked
order polytopes bijectively correspond to bases elements of irreducible
representations; cf.~the discussion in~\cite{ABS,bliem}. Bliem~\cite{bliem}
used partition functions of \emph{chopped and sliced cones} to show that in
the marking $\l$, the dimension of the corresponding irreducible
representation is given by a piecewise quasipolynomial.
Theorem~\ref{thm:piecew_poly} strengthens his result to piecewise polynomials.
Bliem~\cite[Warning~1]{bliem} remarks that his `regions of
quasi-polynomiality' might be too fine in the sense that the quasi-polynomials
for adjacent regions might coincide. This also happens for the piecewise
polynomial described in Theorem~\ref{thm:piecew_poly}. In the simplest case $A
= P$ and $\Omega_{P,A} \equiv 1$. 

\begin{quest}
    What is the coarsest subdivision of $\ordC(A)$ for which
    $\Omega_{P,A}(\l)$ is a piecewise polynomial?
\end{quest}

For this it is necessary to give a combinatorial condition when two adjacent cells of $\T_A$ carry
the same polynomial.

\begin{example}
Consider the following poset $P$ given by its Hasse diagram:

\begin{figure}[h]
\includegraphics[width=1.7cm]{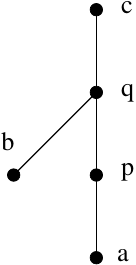}
\end{figure}

Let $A=\{a,b,c\}$ and let $\lambda \colon A \rightarrow \mathbb{Z}$ be an order preserving map. If $\lambda (a)<\lambda (b)<\lambda (c)$ then there are two compatible linear extensions of $P$:
\[
\begin{array}{ccccccccc}
a&\prec &b&\prec &p&\prec &q&\prec &c\\
a&\prec &p&\prec &b&\prec &q&\prec &c\\
\end{array}
\]
The number of lattice points in the corresponding maximal cells of $\mathcal{O}_{P,A}(\lambda)$ are ${\lambda (c) -\lambda (b)+2 \choose 2}$ and $(\lambda (c)-\lambda (b)+1)(\lambda (b) -\lambda (a) +1)$ respectively. Taking into account overcounting we have to substract the number of order preserving extensions of $\lambda$ to $P$ for which $p$ and $b$ have the same value. These correspond to lattice points in the cell given by the chain of order ideals
\[
\{a\}\subset \{a,b,p\} \subset \{a,b,p,q\} \subset \{a,b,p,q,c\}
\] 
and their number is $\lambda (c)-\lambda (b)+1$. In total we have
\[
\Omega _{P,A}(\lambda)={\lambda (c) -\lambda (b)+2 \choose 2}+(\lambda (c)-\lambda (b)+1)(\lambda (b) -\lambda (a) +1)-(\lambda (c)-\lambda (b)+1).
\]
If $\lambda (b)<\lambda (a)<\lambda (c)$ then the only compatible linear extension is
\[
\begin{array}{ccccccccc}
b&\prec &a&\prec &p&\prec &q&\prec &c\\
\end{array}
\]
and 
\[
\Omega _{P,A}(\lambda)={\lambda (c) -\lambda (a)+2 \choose 2}.
\]
\end{example}

\subsection{Chains and Cayley cones}\label{ssec:chains}

Let us consider the special case in which $A \subseteq P$ is a chain. It turns
out that in this case the relation between $\ordC(P)$ and $\ordC(A)$ is very
special. A pointed polyhedral cone $K \subset \R^n$ is called a \Emph{Cayley
cone} over $L$ if there is a linear projection $\pi: K \rightarrow L$ onto a pointed
simplicial cone $L$ such that every ray of $K$ is injectively mapped to a ray
of $L$. In case $K$ is not pointed, then $K \cong K^\prime \times U$ where
$K^\prime$ is pointed and $U$ is a linear space and we require $L \cong
L^\prime \times U$ and $\pi$ is an isomorphism on $U$.  Cayley cones are the
``cone-analogs'' of Cayley configurations/polytopes~\cite[Sect.~9.2]{DLRS}
which are precisely the preimages under $\pi$ of bounded hyperplane sections
$L \cap H$. 

\begin{prop}
    If $A \subseteq P$ is a chain and $\min(P) \cup
    \max(P) \subseteq A$, then $\ordC(P)$ is a Cayley cone over
    $\ordC(A)$.
\end{prop}
\begin{proof}
    The restriction map $\pi(\phi) = \phi|_A$ for $\phi \in \ordC(P)$ is a
    surjective linear projection.  Since $A$ is a chain and $\min(P) \cup
    \max(P) \subseteq A$, $A$ and $P$ are connected posets. The lineality
    spaces are spanned by $1_A$ and $1_P$, respectively, and $\pi$ is an
    isomorphism on lineality spaces. Moreover, $\ordC_0(A) = \ordC(A) / (\R
    \cdot 1_A)$ is linear isomorphic to the cone of order preserving maps $A
    \rightarrow \R_{\ge0}$ which map $\min(A) = \{a_0\}$ to $0$, which shows
    that $\ordC_0(A)$ is simplicial.

    Thus, we only need to check that $\pi : \ordC_0(P) \rightarrow \ordC_0(A)$
    maps rays to rays. It follows from the description of face partitions
    (Proposition~\ref{prop:face_graph}) that
    the rays of $\ordC_0(P)$ are spanned by indicator functions of proper filters. Let $\phi$ be such an indicator function. Then also $\phi|_A : A \rightarrow \{0,1\}$ is an indicator function of a proper filter of $A$ which proves the claim. 
\end{proof}

Here is the main property of Cayley cones that make them an indispensable tool
in the study of mixed subdivisions and mixed volumes.

\begin{prop}
    Let $K$ be a pointed Cayley cone over $L$. Let $r_1,\dots,r_k$ be linearly
    independent generators of $L$ and let $K_i =
    \pi^{-1}(r_i)$ be the fiber over the generator $r_i$. Then for every point $p
    \in L$ we have
    \[
        \pi^{-1}(p) \ = \ \mu_1 K_1 \ + \  \mu_2 K_2 \ + \  \cdots \ + \
        \mu_r K_r
    \]
    where $\mu_1,\mu_2,\dots,\mu_r \ge 0$ are the unique coefficients such that $p =
    \sum_i \mu_ir_i$.
\end{prop}
\begin{proof}
    Let $\tset{s_{ij} \in K }{ 1 \le i \le k, 1 \le j \le m_i }$ be a 
    minimal generating set of $K$ such
    that $\pi(s_{ij}) = r_i$. It follows that $K_i = \conv \tset{ s_{ij} }{ 1
    \le j \le m_i }$. Thus, if $\mu_{ij} \ge 0$ are such that 
    \[
        \sum_{i,j} \mu_{ij} s_{ij} \ \in \ \pi^{-1}(p)
    \]
    then, by the uniqueness of the $\mu_i$, we have $\sum_{j} \mu_{ij} =
    \mu_i$ and $\sum_{j} \mu_{ij} s_{ij} \in \mu_iK_i$.
\end{proof}

If $A = \{ a_0 \prec_P a_1 \prec_P \cdots \prec_P a_k \}$ is a chain, recall
that $\phi^0,\phi^1,\dots,\phi^k : A \rightarrow \{0,1\}$ with
$\phi^i(a_j) = 1$ iff $j \ge i$ is a minimal generating set of $\ordC(A)$. If
$\l : A \rightarrow \R$ is an order preserving map, then unique coordinates of
$\lambda \in \ordC(A)$
with respect to $\{\phi^i\}$ are given by $\mu_0 = \l(a_0)$ and $\mu_i =
\l(a_i) - \l(a_{i-1})$ for $ 1 \le i \le r$. 

\begin{cor}
    Let $P$ be a poset and 
     $A \subseteq P$ a chain such that 
    $\min(P) \cup \max(P) \subseteq A$. Let
    $\Phi_i = \ord_{P,A}(\phi^i)$ for $i = 1,2, \dots, k$. Then for any order
    preserving map $\l: A \rightarrow \R$ we have
    \begin{equation}\label{eqn:ord_cayley}
        \ord_{P,A}(\l) \ = \ \mu_0 1_P \ + \ \mu_1\Phi_1 \ + \ \mu_2\Phi_2 \ + \ \cdots  \
        + \ \mu_k\Phi_k.
    \end{equation}
\end{cor}
This was already observed by Stanley~\cite[Thm.~3.2]{stanley81} and used to
show that the number of order preserving maps extending a given map on a chain
$A \subset P$ satisfy certain log-concavity conditions. This is done by
identifying the numbers as mixed volumes which are calculated from the Cayley
polytope.

In particular, $\Omega_{P,A}(\l)$ counts the number of lattice points in the
Minkowski sum \eqref{eqn:ord_cayley}. It follows from
Theorem~\ref{thm:piecew_poly} and~\eqref{eqn:prod_ehr} that over a maximal
cell $C \in \T_A$, the function $\Omega_{P,A}(\l)$ can be written as a
polynomial $f(\mu)$ in the coordinates $\mu = (\mu_1,\dots,\mu_k)$. The degree
of $f(\mu)$ in every variable $\mu_i$ is given by
\[
    \deg_{\mu_i} f(\mu) \ = \ \dim \Phi^i \ = \ |P \setminus (P_{\preceq
    a_{i-1}} \cup P_{\succeq a_{i}}) |
\]
The degree in $\l_i$ is more difficult to determine. 
\begin{quest}
    What is $\deg_{\l_i} \Omega_{P,A}(\l)$ in terms of the combinatorics of
    $P$?
\end{quest}

If $A \subseteq P$ is a
chain with minimum $a_0$ and maximum $a_k$, then the degree of $\l_0$ and
$\l_k$ agrees with $\mu_1$ and $\mu_k$.
A related situation is implicitly treated in Fischer~\cite{fischer06}: The
number $\alpha(n;k_1,k_2,\dots,k_n)$ of monotone triangles with bottom row $\k
= (k_1 \le k_2 \le \cdots \le k_n)$ is a polynomial in $\k$ and is of degree
$n-1$ in every variable $k_i$.  In Section~\ref{sec:MT}, it is shown that
$\alpha(n;\k)$ is essentially the number of integer-valued order preserving
extensions from a particular poset with some extra conditions (i.e.\ certain
faces of the marked order polytope are excluded).  However, it appears that
these extra condition do not influence the degree.

\subsection{Combinatorial reciprocity} \label{ssec:reciprocity}
For a special choice of $A$, we recover
the classical order polytope.

\begin{example}[Order polytopes]
    Let $P^\prime$ be the result of adjoining a minimum $\hat{0}$ and
    maximum $\hat{1}$ to $P$. Let $A = \{ \hat{0}, \hat{1} \}$ and for $n > 0$
    let $\l_n : A \rightarrow \Z$ be the order preserving map with
    $\l_n(\hat{0}) = 1$ and $\l_n(\hat{1}) = n$. Then
    $\Omega_{P^\prime,A}(\l_n) = \Omega_P(n)$ is the order polynomial of $P$
    which counts the number of order preserving maps from $P$ to $[n]$.
    Equivalently, $\Omega_{P^\prime,A}(\l_n)$ equals the Ehrhart polynomial 
    of the order polytope $\ordC(P) \cap [0,1]^P$ evaluated at $n-1$.
    Ehrhart-Macdonald Reciprocity (see for example~\cite[Thm.~4.1]{BR}) then yields that 
    \[
        (-1)^{|P|}\,\Omega_P(-n) \ = \
        (-1)^{\dim\,\ord_{P^\prime,A}(\l_n)}\, \Omega_{P^\prime,A}(\l_{-n})
    \]
    equals the number of strictly order preserving maps into $[n]$.
\end{example}

We wish to extend this combinatorial reciprocity to our more general setting.
For that we say that an extension $\widehat{\l} : P \rightarrow \R$ of $\l$ is
\Emph{strict} if $\widehat{\l}(p) = \widehat{\l}(q)$ and $p \prec q$ implies
that $a \preceq p \prec q \preceq b$ for
some $a,b \in A$ with $\l(a) = \l(b)$. 

\begin{thm}\label{thm:reciprocity}
    Let $A \subset P$ be a pair of posets with $\min(P)\cup\max(P) \subseteq
    A$. If $\l : A \rightarrow \Z$ is an order preserving map, then
    \[
        (-1)^{\dim\,\ord_{P,A}(\l)}\,\Omega_{P,A}(-\l) 
    \]
    equals the number of strict order preserving extensions of $\l$.
\end{thm}

Note that if $F(\I^A)$ is the unique cell of the subdivision of $\ordC(A)$ that
contains $\l$ in the relative interior, then $\Omega_{P,A}(\l)$ is the evaluation of a polynomial and
it is this polynomial that is evaluated at $-\l$ in the course of the theorem
above. From the geometric point of view,
$(-1)^{\dim\,\ord_{P,A}(\l)}\,\Omega_{P,A}(-\l)$ counts the number of lattice
points in the relative interior of $\ord_{P,A}(\l)$. This is reminiscent of
Ehrhart-Macdonald reciprocity and in fact follows from it.

\begin{proof}
    For fixed $\l$, let $\I^A$ such that $\l \in \relint F(\I^A)$. Then
    $\Omega_{P,A}$ restricted to $\relint F(\I^A)$ is given by some
    polynomial $P(\mathbf{x}) \in \R[x_a : a \in A]$. For $n \in \Z_{>0}$, we
    have that $n\l \in \relint F(\I^A)$ and thus $\Omega_{P,A}(n\l) =
    P(n\l)$. As $\Omega_{P,A}(n\l)$ equals the number of lattice points in
    $n\,
    \ord_{P,A}(\l)$, it follows that $P(n\l)$ is the Ehrhart polynomial of
    $\ord_{P,A}(\l)$.  Now, Ehrhart-Macdonald reciprocity implies that the
    number of points in the relative interior of $\ord_{P,A}(\l)$  equals
    \[
        (-1)^d\,\Ehr(\ord_{P,A}(\l),-1) \ = \ (-1)^d P(-\l) \ = \ (-1)^d
        \Omega_{P,A}(-\l).
    \] 
    where $d = \dim\,\ord_{P,A}(\l)$.
\end{proof}

\subsection{Marked chain polytopes}\label{ssec:odds}

Let us close by \Emph{transferring} our results to the marked chain polytopes
of Ardila, Bliem, and Salazar~\cite{ABS}. To that end we write $\phi(C) = \sum
\tset{ \phi(c) }{ c \in C }$ for a subset $C \subseteq P$ and $\phi : P
\rightarrow \R$.  For a pair of posets $A \subset P$ and an order preserving
map $\l : A \rightarrow \R$, the \Emph{marked chain polytope} is the convex
polytope
\[
    \mathcal{C}_{P,A}(\l) \ = \ \set{ \phi \in \R^P_{\geq 0} }{ \phi(C) \ \le \ 
    \l(b) - \l(a) \text{ for every chain $C \subseteq [a,b]$ and $a,b \in A$
    }}
\]
The unmarked version of the chain polytope was introduced in \cite{stanley86}
to show that certain invariants of $P$ (such as $\Omega_P(n)$) only depend on
the comparability graph of $P$. The marked chain polytopes were introduced
in~\cite{ABS} in connection with representation theory.  Stanley defined a
lattice preserving, piecewise linear map from the order polytope to the chain
polytope and this transfer map was extended in~\cite{ABS} to relate the
arithmetic of marked order polytope and marked chain polytopes. Thus,
appealing to Theorem~3.4 of~\cite{ABS} proves
\begin{cor}
    For a pair of posets $A \subset P$, $\min(P) \cup
    \max(P) \subseteq A$, the function
    \[
        \l \ \mapsto \ |\mathcal{C}_{P,A}(\l) \ \cap \ \Z^P|
    \]
    is a piecewise polynomial over $\ordC(A) \cap \Z^A$ and evaluating at
    $-\l$ equals $(-1)^{\dim\,\mathcal{C}_{P,A}(\l)}$ times the number of
    lattice points in the relative interior of $\mathcal{C}_{P,A}(\l)$.
\end{cor}

\section{Monotone triangle reciprocity}\label{sec:MT}

A \Emph{monotone triangle} (MT, for short) of order $n$, as exemplified
in~\eqref{eqn:MTexample}, is a triangular array of integers $a = (a_{i,j})_{1
\le j \le i \le n} \in \Z$ such that the entries
\begin{compactenum}[(M1)]
    \item weakly increase along the north-east direction: $a_{i,j}
        \le a_{i-1,j}$ for all $1 \le j < i \le n$,
    \item weakly increase along the south-east direction: $a_{i,j}
        \le a_{i+1,j+1}$ for all $1 \le j \le i < n$, and
    \item \label{not_bad} strictly increase in the rows: $a_{i,j}
        < a_{i,j+1}$ for all $1 \le j < i < n$.
\end{compactenum}

The number of monotone triangles with fixed bottom row $\k = (k_1 \le k_2 \le
\cdots \le k_n)$ is finite and denoted by $\alpha(n;k_1, k_2, \dots, k_n)$. 
Monotone triangles originated in the study of alternating sign
matrices~\cite{MRR} where it was shown that alternating sign matrices of order
$n$ exactly correspond to monotone triangles with bottom row $(1,2,\dots,n)$.
The study of enumerative properties of monotone triangles with general bottom
row was initiated in~\cite{fischer06} where it was shown that
$\alpha(n;k_1,k_2,\dots,k_n)$ is a polynomial in the strictly increasing
arguments. Note that our definition of a monotone triangle slightly differs
from that of Fischer~\cite{fischer06} inasmuch that we do not require that the
bottom row is strictly increasing.

More precisely, there is a polynomial that agrees with $\alpha(n;\k)$ for
increasing $\k = (k_1 \leq k_2 \leq \cdots \leq k_n)$ and, by abuse of
notation, we identify $\alpha(n;\k)$ with this polynomial.  As a polynomial,
$\alpha(n;\k)$ admits evaluations at arbitrary $\k \in \Z^n$ and it is natural
to ask if there are domains for which the values $\alpha(n;\k)$ have
combinatorial significance. An interpretation for the values of $\alpha$ at
weakly decreasing arguments was given by Fischer and Riegler \cite{FR} in
terms of signed
enumeration of so called decreasing monotone triangles. A \Emph{decreasing
monotone triangle} (DMT) is again a triangular array $b = (b_{i,j})_{1 \le j
\le i \le n} \in \Z$ such that

\begin{compactenum}[(W1)]
    \item the entries weakly decrease along the north-east direction: $b_{i,j}
        \ge b_{i-1,j}$ for \mbox{$1 \le j < i \le n$},

    \item the entries weakly decrease along the south-east direction: $b_{i,j}
        \ge b_{i+1,j+1}$ for \mbox{$1 \le j \le i < n$, }

    \item there are no three identical entries per row, and
    \item two consecutive rows do not contain the same integer exactly
        once.
\end{compactenum}
An example of a DMT is 
\begin{equation}\label{eqn:WTexample}
    \newcommand\uZ{{2}}
    \newcommand\uD{{3}}
    \begin{array}{ccccccccc}
          & &   & & 3 & &    & &  \\ 
          & &   &3&   &3&    & &  \\ 
          \cline{4-6}
          & & 4 & & 3 & & 3  & &  \\
          &4&   &4&   &3&    &2&  \\
          \cline{2-4}
        4 & & 4 & & 3 & & 3 & & 1\\
          \cline{1-3}\cline{5-7}
    \end{array}
\end{equation}
The  collection of decreasing monotone triangles with bottom row $\k =
(k_1 \ge k_2 \ge \cdots \ge k_n) \in \Z^n$ is denoted by $\mathcal{W}_n(\k)$.
For a DMT $b$, two adjacent and identical elements in a row are called a
\Emph{duplicate-descendant} if they are either in the last row or the row
below contains exactly the same pair. In the example, the
duplicate-descendants are underlined. The number of duplicate-descendants of
$b$ is denoted by $\dd(b)$.  The precise reciprocity statement now is

\begin{thm}[{\cite[Thm.~1]{FR}}] \label{thm:MT_reciprocity}
    For weakly decreasing integers $\k = (k_1 \ge k_2 \ge \cdots \ge k_n)$ we
    have
    \[
    \alpha(n;k_1,k_2,\dots,k_n) \ = \ (-1)^{\tbinom{n}{2}} \sum_{b \in \mathcal{W}_n(\k)}
    (-1)^{\dd(b)}.
        \]
\end{thm}

In this section we give a geometric proof for the result above by relating
(decreasing) monotone triangles to special order preserving maps. 
A \Emph{Gelfand-Tsetlin poset} $\GT_n$ of
order $n$ is the poset on $\tset{ (i,j) \in \Z^2 }{ 1 \le j \le i \le n}$ with
order relation 
\[
    (i,j) \preceq_{\GT_n} (k,l) \quad :\Longleftrightarrow \quad k-i \le l-j \text{
    and } j\le l.
\]
The Hasse diagram for $\GT_n$ is given in Figure~\ref{fig:GTn}.
\begin{figure}[b]
    \begin{overpic}[width=4cm]%
        {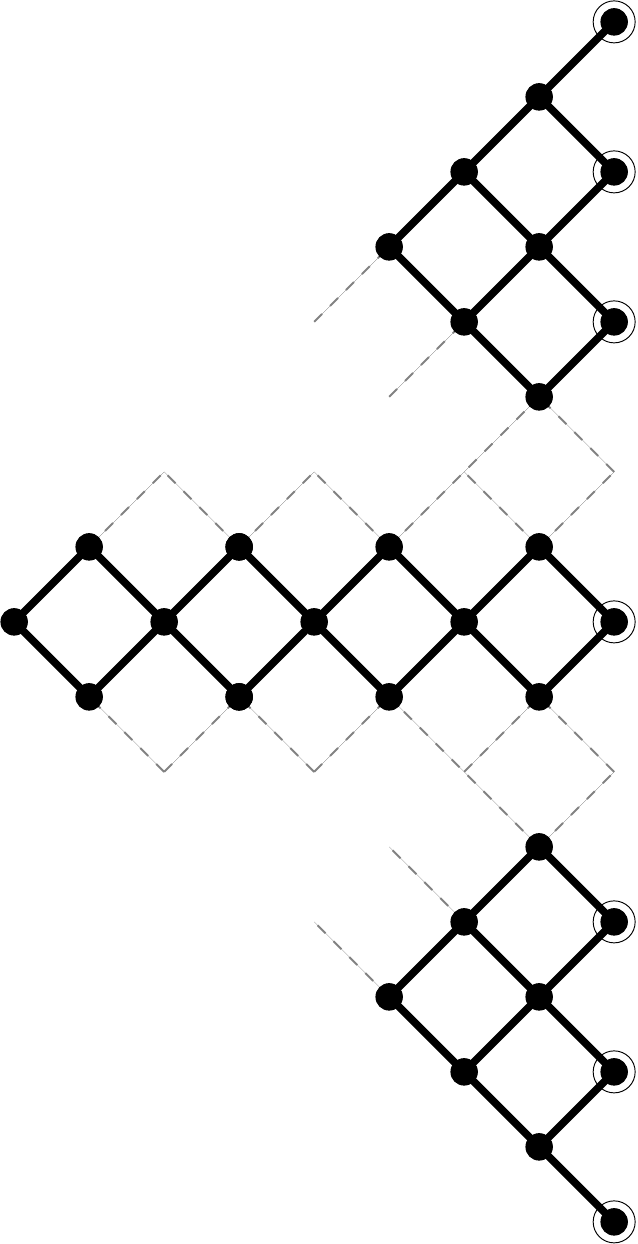}
        \put(52,1){$\scriptsize(n,1)      = \kappa_1    $}
        \put(52,13){$\scriptsize(n,2)   = \kappa_2    $}
        \put(52,25){$\scriptsize(n,3)   = \kappa_3    $}
        \put(52,49){$\scriptsize(n,i) = \kappa_i    $}
        \put(52,73){$\scriptsize(n,n-2)   = \kappa_{n-2}$}
        \put(52,85){$\scriptsize(n,n-1)   = \kappa_{n-1}$}
        \put(52,97){$\scriptsize(n,n)     = \kappa_n    $}
    \end{overpic}
    \caption{Hasse diagram for the Gelfand-Tsetlin poset of order $n$ (in
    solid black). }
    \label{fig:GTn}
\end{figure}
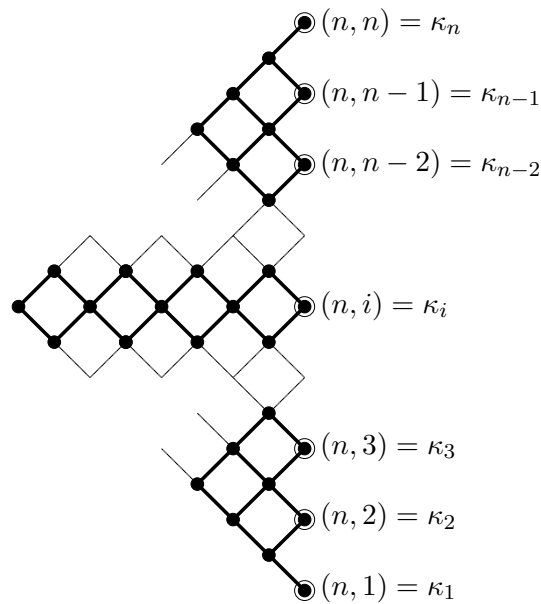
Throughout, we let $A = \{ \kappa_1, \kappa_2, \dots,\kappa_n\} \subset \GT_n$
be the $n$-chain of elements $\kappa_j = (n,j)$ with $1 \le j \le n$, depicted
by the circled elements in Figure~\ref{fig:GTn}.  An increasing sequence $k =
(k_1 \le k_2 \le \cdots \le k_n)$ corresponds to a order preserving
map $k: A \rightarrow \Z$ by setting $k(\kappa_i) = k_i$. We call an order preserving map $a : \GT_n \rightarrow \Z$ a \Emph{weak monotone triangle}
(WMT) (also known as \Emph{Gelfand-Tsetlin pattern}). Here is the main observation.

\begin{obs}
    The collection of monotone triangles $a = (a_{ij})_{1\le j \le i \le n} \in \Z$ for given
    bottom row $\k = (k_1 \le k_2 \le \cdots \le k_n) \in \Z^n$ bijectively
    correspond to integral-valued order preserving maps $a : \GT_n \rightarrow
    \Z$ extending $\k : A \rightarrow \Z$ and such that $a_{i,j}  < a_{i,j+1}$
    for all $1 \le j < i < n$.
\end{obs}

To put this initial observation to good use, we pass to \Emph{real-valued} order
preserving maps and we call an order preserving map $a : \GT_n \rightarrow
\R$ extending $\k$ a \Emph{monotone triangle} if it satisfies (M3). Hence, the
monotone triangles with bottom row $\k$ form a special subset of the marked
order polytope for $\GT_n$
\[
    \OGT_n(\k) \ := \ \ord_{\GT_n,A}(\k). 
\]
Let us denote by $\bad_n = \tset{ (i,j) }{ 1 \le j < i < n }$ and for $(i,j)
\in \bad_n$ define
\[
    Q_{ij} \ = \ \set{ a \in \ordC(\GT_n) }{ a_{i,j} = a_{i,j+1} }
\]
as the set of real-valued weak monotone triangles which fail (M3)
non-exclusively at position $(i,j)$. The Hasse diagram of the face partition
$G_{ij} = G(Q_{ij})$ of $Q_{ij}$ is a \Emph{diamond} in $\GT_n$:
\begin{center}
    \begin{overpic}[angle=45,width=2cm]%
        {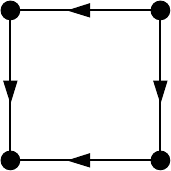}
        \put(30,105){$\scriptscriptstyle(i,j+1)$}
        \put(38,-7){$\scriptscriptstyle(i,j)$}
        \put(-40,48){$\scriptscriptstyle(i-1,j)$}
        \put(-100,48){$G_{ij} = $}
        \put(100,48){$\scriptscriptstyle(i+1,j+1)$}
    \end{overpic}
\end{center}
It is easy to see that $G_{ij}$ is a compatible face partition for strictly
increasing bottom row $\k$ and together with a count of parameters yields the
following geometric result.

\begin{prop}\label{prop:bad_face}
    Let $\k = (k_1 < k_2 < \cdots < k_n)$. For $(i,j) \in
    \bad_n$, the set $Q_{ij}(\k) \subseteq \OGT_n(\k)$ is a face of
    codimension $3$.
\end{prop}

This yields a geometric perspective on monotone triangles.

\begin{cor}\label{cor:mt_geometry}
    For $\k = (k_1 \le k_2 \le \cdots \le k_n)$, the set of monotone triangles
    with bottom row $\k$ are precisely the lattice points in
    \begin{equation}\label{eqn:mt_geometry}
     \OGT_n(\k) \ \setminus \
    \bigcup_{(i,j)\in\bad_n} Q_{ij}(\k).
    \end{equation}
\end{cor}

Notice that if $\k$ contains three identical elements $k_{j} = k_{j+1} =
k_{j+2}$, then $\OGT_n(\k) \subseteq Q_{n-1,j}(\k)$ and the above set
is empty. Hence, the number of monotone triangles with bottom row $\k$ can
only be non-zero if $\k$ contains at most pairs of identical elements. 

Corollary~\ref{cor:mt_geometry} allows us to write $\alpha(n;\k)$ as a
polynomial by inclusion-exclusion on the set of faces $\tset{ Q_{ij}(\k) }{
(i,j) \in \bad_n }$. More refined, we will consider the poset of non-empty
intersections of faces of the form $Q_{ij}$ and obtain $\alpha(n;\k)$ as a
polynomial by M\"obius inversion on that poset. This will be relatively easy
once we have a characterization of the face partitions of such finite
intersections. For that we call an subposet $G \subseteq \GT_n$ a
\Emph{diamond poset} if the Hasse diagram of $G$ is a union of graphs
$G_{i,j}$. In addition, we call a diamond poset \Emph{closed} if
$G_{i,j},G_{i,j+1} \subset G$ implies $G_{i-1,j},G_{i+1,j+1} \subset G$. That
is,
\begin{center}
    \begin{overpic}[angle=45,width=3cm]%
        {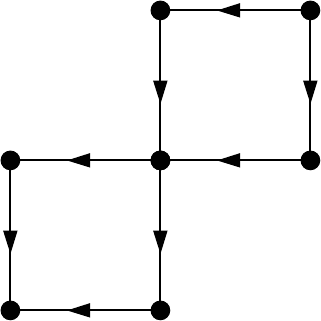}
        \put(35,70){$\scriptstyle G_{i,j+1}$}
        \put(40,25){$\scriptstyle G_{i,j}$}
        \put(60,00){$\scriptscriptstyle(i,j)$}
        \put(60,48){$\scriptscriptstyle(i,j+1)$}
        \put(110,48){$\displaystyle \in G \quad \Longrightarrow \quad $}
    \end{overpic}
    \hspace{2cm}
    \begin{overpic}[angle=45,width=3cm]%
        {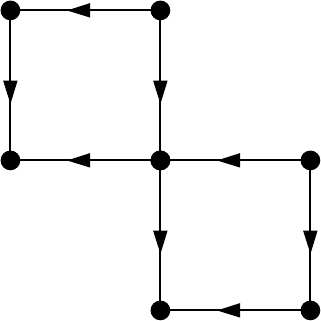}
        \put(10,15){$\scriptscriptstyle(i-1,j)$}
        \put(55,15){$\scriptscriptstyle(i+1,j+1)$}
        \put(110,48){$\displaystyle \in G$}
    \end{overpic}
\end{center}

\begin{lem}\label{lem:closed_diamond}
    Let $F
    \subseteq \ordC(\GT_n)$ be a non-empty face. Then 
    \[
        F \ = \ \bigcap_{(i,j) \in I} Q_{ij}
    \]
    for some $I \subseteq \bad_n$ if and only if $G(F)$ is a closed
    diamond poset.
\end{lem}
\begin{proof}
    The face $F$ is exactly the intersection of all facets for which the
    corresponding cover relation is in $G(F)$. If $G(F)$ is a closed diamond
    poset, then every cover relation is contained in at least one diamond and
    hence $F$ is exactly the intersection of all $Q_{ij}$ for which $G_{ij}
    \subseteq G(F)$.

    For the converse, we can assume that $G=G(F)$ is
    connected and we let $G^\prime = \bigcup\{ G_{ij} : F \subseteq Q_{ij} \}$
    be the largest diamond poset contained in $G$. If $G \not= G^\prime$, then by Proposition \ref{prop:face_graph}
    there is a non-trivial directed path $P = p_0p_1\dots p_k$ that meets $G^\prime$
    only in a  connected component containing $p_0$ and $p_k$.  In
    particular no edge of $P$ is contained in a diamond of $G$ and,
    furthermore, $P$ cannot contain vertices $(i,j)$ and $(i,j+1)$.  Indeed,
    by Proposition~\ref{prop:face_graph}, this would imply that $G_{ij}
    \subset G^\prime$ which contradicts $P \cap G^\prime = \{p_0,p_k\}$. It
    follows that $c = p_{i+1} - p_{i} \in \Z^2$ is a constant direction for
    all $i = 0,1,...,k-1$.
    
    Let us assume that $c = (+1,0)$. Thus, every vertex $p_h$ along $P$ has
    constant second coordinate $\ell = (p_h)_2$. Let $R$ %
    be an undirected(!) path connecting $p_0$ and $p_k$ in $G^\prime$
    such that 
    \[
        \rho(R) \ = \ \sum_{r \in R} |r_2 - \ell|
    \]
    is minimal. Such a path exists as $p_0$ and $p_k$ are in the same
    connected component of the underlying undirected graph of $G^\prime$, and $\rho(R) > 0$.
     (Indeed, we have $\rho(R) = 0$ iff $R = P$ after orienting
    edges). But then $R$ contains a sequence of vertices $(i,j),(i-1,j),(i,j+1)$ with $j<l$ or $(i,j),(i+1,j+1),(i,j+1)$ with $j\geq l$ and the value
    of $\rho(R)$ can be reduced by rerouting along $G_{i,j}$. 
    \begin{center}
        \begin{overpic}[width=1.5cm,angle=45]%
            {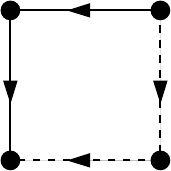}
            \put(30,105){$\scriptscriptstyle(i,j+1)$}
            \put(38,-10){$\scriptscriptstyle(i,j)$}
            \put(-40,47){$\scriptscriptstyle(i-1,j)$}
            \put(100,47){$\scriptscriptstyle(i+1,j+1)$}
        \end{overpic}
        \hspace{3cm}
        \begin{overpic}[width=1.5cm,angle=315]%
            {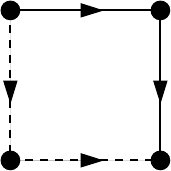}
            \put(30,105){$\scriptscriptstyle(i,j+1)$}
            \put(38,-10){$\scriptscriptstyle(i,j)$}
            \put(-40,47){$\scriptscriptstyle(i-1,j)$}
            \put(100,47){$\scriptscriptstyle(i+1,j+1)$}
        \end{overpic}
    \end{center}
    Hence, by contradiction, $R = P$ and $G = G^\prime$.   
\end{proof}

Let us define $\Q$ as the set of all closed diamond subposets of $\GT_n$
ordered by reverse inclusion. In light of the above lemma, we have
\[
    \Q \ \cong \ \biggl\{ \bigcap_{(i,j)\in I} Q_{ij} \;:\; I
    \subseteq \bad_n \biggr\}
\]
is a meet-semilattice with greatest element $\hat{1} = \hat{1}_\Q :=
\emptyset$ corresponding to $\ordC(\GT_n)$.  The M\"obius function of $\Q$ can
now be dealt with in the language of diamond posets. Let us write
\[
I(G) \ = \ \set{ (i,j) \in \bad_n }{ G_{ij} \subseteq G}
\]
for $G \in \Q$.

\begin{lem} \label{lem:mobius}
    Let $G \in \Q$ and $I = I(G)$.  Then
    \[
        \mu_\Q(G,\hat{1}) \ = \ 
        \begin{cases}
            0, & \text{if } (i,j),(i,j+1) \in I\\
            (-1)^{|I|}, &  \text{otherwise.}
        \end{cases}
    \]
\end{lem}
\begin{proof}
    \newcommand\A{\mathcal{A}}
    Let $\A $ be the collection of \Emph{atoms} of the interval
    $[G,\hat{1}]_\Q$, that is, the elements of $\Q$ covering $G$. To prove the
    first claim, we will use the Crosscut Theorem~\cite[Sec.~3.1.9]{rota}
    \[
        \mu_\Q(G,\hat{1}) \ = \ N_0 \ - \ N_1 \ + \ \cdots \ +(-1)^i N_i 
    \]
    where $N_k$ is the number of $k$-element subsets $S \subseteq \mathcal{A}$
    such that $\hat{1}$ is the smallest joint upper bound for the elements in
    $S$. Now if there is some $Q \prec \hat{1}_\Q$ such that every $H \in \A$
    is smaller than $Q$, then this implies $N_k = 0$ for all $k$ and the claim
    follows. 
    
    To that end, let $(i_0,j_0) \in I(G)$ with $(i_0+1,j_0),(i_0+1,j_0+1) \in
    I(G)$ and $i_0$ minimal. We claim that $(i_0,j_0) \in I(H)$ for every $H
    \in \mathcal{A}$. Indeed, assume that $(i_0,j_0) \not\in I(H)$. By
    Lemma~\ref{lem:closed_diamond}, we have that $H \cup
    G_{i_0,j_0}$ is a diamond poset but not closed, as $H\in \mathcal{A}$ by
    assumption. This forces $G_{i_0,j_{0}-1}$ or $G_{i_0,j_{0}+1}$ to be in
    $G$, and establishing then the  closedness condition has to
    introduce some $G_{i,j}\subseteq G$ with $i < i_0$.  However, this contradicts the
    choice of $i_0$ and we can take $Q = Q_{i_0j_0}$.

    For the other case, observe that the closedness condition for $G$ is
    vacuous. This stays true for every diamond subposet which are in
    bijection to the subsets of $I(G)$. Hence $[G,\hat{1}]_\Q$ is isomorphic
    to the boolean lattice on $|I(G)|$ elements.
\end{proof}

This yields a partial explanation of condition (W3): A weak monotone triangle
$a : \GT_n \rightarrow \R $ with strictly increasing bottom row satisfies (W3)
and (W4) if and only if $a \in \relint F$ for some face $F$ with $G = G(F) \in
\Q$ and 
$\mu_\Q(F,\hat{1}) \not= 0$. For that reason, let us define 
$\Qess \subseteq \Q$ as the \Emph{essential} subposet of $\Q$ with
\[
    \Qess \ = \ \set{ G \in \Q }{\mu_\Q(G,\hat{1}) \not= 0}
\]
Hence, we can identify $\Qess$ with
the collection of closed diamond posets $G$ of $\GT_n$ such that $G_{i,j} \cup
G_{i,j+1} \not\subseteq G$.
In particular, $\hat{1} \in \Qess$ and from the definition of M\"obius
functions it follows that
$\mu_{\Qess}(G,\hat{1}) = \mu_\Q(G,\hat{1})$ for all $G \in \Qess$.

With that knowledge, we can now write the number of lattice points in
\eqref{eqn:mt_geometry} as a polynomial in $\k$. For the sake of clarity, let
us emphasize that the combinatorics of $\Q_{\GT_n,A}(\k)$ is independent of
the actual choice of a strictly order preserving map $\k : A \rightarrow \R$.
In this case, every $G \in \Qess$ is a compatible face partition of a distinct
face of $\OGT_n(\k)$ which we can identify with 
the marked order polytope $\ord_{\GT_n/G,A/G}(\k)$.

\begin{thm}\label{thm:mt_polynomial}
    For $\k = (k_1 \le k_2 \le \cdots \le k_n)$, the number of monotone
    triangles with bottom row $\k$ is given by
    \[
    \alpha(n;\k) \ = \ \sum_{G \in \Qess} (-1)^{|I(G)|} \,
    \Omega_{\GT_n/G,A/G}(\k),
    \]
    and thus is a polynomial.  In particular, $\alpha(n;\k) = 0$ whenever $k_j =
    k_{j+1} = k_{j+2}$.
\end{thm}
\begin{proof}
    If $\k$ is strictly order preserving, then the above formula is exactly
    the M\"obius inversion of the function $f_G(\k) =
    \Omega_{\GT_n/G,A/G}(\k)$ for $G \in \Qess$ by
    Corollary~\ref{cor:mt_geometry} and Lemmas~\ref{lem:closed_diamond}
    and~\ref{lem:mobius}.

    If $\k$ has two, but no three identical entries, then $G \in \Qess$ is not
    compatible with $\k$ but can be completed to a compatible face partition
    $\bar{G}$. It is easy to see that $\bar{G}$ arises from $G$ by only adding
    the cover relations $(n,j) \prec_{\GT_n} (n-1,j)$ and $ (n-1,j)
    \prec_{\GT_n} (n,j+1)$ for every $1 \le j < n$ with $k_j = k_{j+1}$. The
    map $G \mapsto \bar{G}$ is injective on $\Qess$ and the image is a poset
    under reverse inclusion isomorphic to $\Qess$. Hence, the above formula
    counts the number of lattice points in \eqref{eqn:mt_geometry}.
    
    If $\k$ has three identical entries, then \eqref{eqn:mt_geometry} is the
    empty set and $\alpha(n;\k) = 0$. Consequently, we have to show that the
    right hand side is also identically zero for all such $\k$. It suffices to
    assume that $\k$ has exactly three identical entries as every bottom row with more than three identical elements belongs to the boundary of some cell for which the interior consists of bottom rows with exactly three identical elements. So, let us assume that $k_j = k_{j+1} = k_{j+2}$ are
    the only equalities for $\k$. Let $G \in \Qess$ and $\bar{G}$ its
    completion to a face partition compatible with $\k$. Then
    $\Omega_{\GT_n/\bar{G},A/\bar{G}}(\k)$ appears in the sum on the right hand side with multiplicity
    \[
        \sum\set{(-1)^{|I(H)|} }{ H \in \Qess, \bar{H} = \bar{G} }.
    \]
    For any such $H$, let $(i,j) \in \bad_n$ be the lexicographic smallest
    such that $G_{i+1,j} \cup G_{i+1,j+1} \subseteq \bar{H} = \bar{G}$ (existence follows from $k_j =
    k_{j+1} = k_{j+2}$).
    Hence $G_{i,j} \subseteq \bar{H}$ by closedness. We distinguish two
    cases:
    \begin{compactenum}[1.]
    \item Assume that $G_{ij} \subseteq H$, then the largest diamond subposet
    $H^\prime \subset H$ not containing $G_{ij}$ is closed as $H \in \Qess$,
    and $\bar{H^\prime} = \bar{G}$ as $G_{i+1,j} \cup G_{i+1,j+1} \subseteq \bar{H}$.
    \item If $G_{ij} \not\subseteq H$, then set $H^\prime = H \cup G_{ij}$. By
    the minimality of $(i,j)$ we have that $H^\prime$ is closed diamond and 
    $\bar{H^\prime} = \bar{G}$.
    \end{compactenum}
    This defines a perfect matching on $\set{H\in\Qess}{\bar{H}=\bar{G}}$ and
    $|I(H)| = |I(H^\prime)| \pm 1$ shows that the multiplicity of
    $\Omega_{\GT_n/\bar{G},A/\bar{G}}(\k)$ is zero.
\end{proof}

Coming back to the reciprocity statement for monotone triangles, we note that
$b = (b_{ij})_{1 \le j \le i \le n}$ is a decreasing monotone triangle
if and only if $-b : \GT_n \rightarrow \R$ is a weak monotone triangle
satisfying (W3) and (W4). 

\begin{prop}\label{prop:open_faces}
    Let $a = (a_{ij})_{1 \le j \le i \le n} \in \Z$ be a weak monotone 
    triangle with bottom row $\k = (k_1 \le k_2 \le \cdots \le k_n)$ with no three identical elements. 
    Then $-a$ is a DMT with bottom row $-\k$ if and only if there is a unique $G \in
    \Qess$ with corresponding face $F \subseteq \OGT_n(\k)$ such that
    $a \in \relint F$.
\end{prop}

\begin{proof}
    Let $F$ be the face of $\OGT_n(\k)$ that has $a$ in the relative interior
    and let $G^\prime = G(F)$ be its compatible face partition. If $\k$ is not
    strictly increasing, then $G^\prime$ contains cover relations that reach
    into $A$.  Let $G \subseteq G^\prime$ be the subposet which arises by
    deleting those which are not contained in a diamond. Then $G$ is a face
    partition and $\Ext_{\GT_n,A}(\l) \cap F_{\GT_n}(G) = F$.

    Now (W4) is equivalent to the condition that every cover relation in $G$
    is contained in a diamond. Otherwise there are indices $(i,j),(i+1,k) \in
    \bad_n$ with $k
    \in \{j,j+1\}$ such that $b_{i,j} = b_{i+1,k}$ and $b_{i,j-1} < b_{i,j} <
    b_{i,j+1}$  and $b_{i,k-1} < b_{i,k} < b_{i,k+1}$ which contradicts (W4).
    Since $\k$ does not contain three identical elements, $G$ is the unique
    diamond poset that gives rise to $F$.  Moreover, $G \in \Qess$ if and only
    if every point in the relative interior of $F$ satisfies (W3).
\end{proof}

Let us extend the notion of duplicate-descendants to real-valued weak monotone
triangles satisfying (W3) and define $\dd(F)$ for a non-empty face $F \subseteq
\OGT_n(\k)$ as the number of duplicate-descendants for an arbitrary $a
\in \relint F$. 

\begin{lem}\label{lem:codim_mod2}
    Let $\k = (k_1 \le k_2 \le \cdots \le k_n)$ with no three identical
    elements and let %
    $m$ be the number of pairs of identical elements. 
    Let $G \in \Qess$ with corresponding face $F
    \subseteq \OGT_n(\k)$.  Then
    \[
    |I(G)| \ + \ \codim F  + m \ \equiv \ \dd(F) \quad \mod\,2
    \]
\end{lem}
\begin{proof}
    We induct on $l = |I(G)|$. For $l = 0$, we have $F =
    \OGT_n(\k)$ which is of codimension $0$ and $\dd(F) = m$ by
    definition.
    
    For $l > 0$ there is a diamond $G_{ij} \subseteq G$ which shares at most
    one edge with another diamond or a ``half-diamond'' coming from a pair of
    equal numbers at the bottom row. Let $G^\prime \subset G$ be the largest
    diamond poset not containing $G_{ij}$ and let $F^\prime$ be the
    corresponding face. By induction, the claim holds for $G^\prime$ and
    $|I(G)| = |I(G^\prime)| + 1$.

    If $G_{ij} \cap G(F^\prime)$ does not contain an edge, then $\dd(F) =
    \dd(F^\prime)$ and $\codim F = \codim F^\prime + 3$.  In the remaining
    case, $G_{ij}$ shares exactly one edge with $G(F^\prime)$ and
    thus $\dd(F) = \dd(F^\prime) + 1$. On the other hand, adding $G_{ij}$ to
    $G(F^\prime)$ binds two degrees of freedom and $\codim F = \codim F^\prime +
    2$.
\end{proof}

\begin{proof}[Proof of Theorem~\ref{thm:MT_reciprocity}]
    By Theorem~\ref{thm:mt_polynomial}, $\alpha \equiv 0$ restricted to
    the set of order preserving maps $-\k : A \rightarrow \Z$ with three identical
    entries. As $\alpha$ is a polynomial, it follows that this extends to
    $\alpha(n;\k)$. This proves the claim in this case as $\mathcal{W}_n(\k) =
    \emptyset$.

    Let us assume that $\k$ has $m$ pairs of identical elements. Then $\dim
    \OGT_n(-\k) = \tbinom{n}{2} - m$.  For \mbox{$G \in \Qess$} let us denote
    by $F_G(-\k)$ the corresponding non-empty face of $\OGT_n(-\k)$.  By
    Theorem~\ref{thm:mt_polynomial} and Theorem~\ref{thm:reciprocity}, we have
    \[
    \alpha(n;\k) \ = \ (-1)^{\tbinom{n}{2}}
        \sum_{G \in \Qess} (-1)^{|I(G)| + m + \codim F_G(-\k)} | \relint
        F_G(-\k) \cap \Z^{\GT_n}| 
    \]
    where we use $\codim F_G(-\k) = \tbinom{n}{2} -m - \dim F_G(-\k)$. The claim
    now follows from Proposition~\ref{prop:open_faces} and Lemma~\ref{lem:codim_mod2}.
\end{proof}

\section{Extending partial graph colorings}\label{sec:color}

Let $G=(V,E)$ be a graph and $k$ a positive integer. A \Emph{$k$-coloring} of
$G$ is simply a map $c\colon V\rightarrow [k]$. The coloring is called
\Emph{proper} if $c(u) \not= c(v)$ for every edge $uv \in E$.  It is well-known
that the number of proper $k$-colorings of $G$ is given by a polynomial in
$k$, the \Emph{chromatic polynomial} $\chi_G(k)$.  Generalizing these notions,
Murty and Herzberg \cite{herzberg} considered the problem of counting
extensions of partial colorings of $G$.
 For a given subset $A \subseteq V$ and a partial
coloring $c : A \rightarrow [k]$ an \Emph{extension} of $c$ of size $m$ is an
$m$-coloring $\widehat{c} : V \rightarrow [m]$ such that $\widehat{c}(a) =
c(a)$ for all $a \in A$. If $\widehat{c}$ is moreover a proper coloring, then
$\widehat{c}$ is called a proper extension. Such extensions only exist for $m
\ge k$.

\begin{thm}[{\cite[Thm.~1]{herzberg}}] \label{thm:col_ext}
    Let $G = (V,E)$ be a graph and $c : A \rightarrow [k]$ a partial coloring
    for $A \subseteq V$. Then either there are no proper extensions or there
    is a polynomial $\chi_{G,c}(m)$ of degree $|V| - |A|$ such that 
    \[
        \chi_{G,c}(m) \ = \ \#\set{ \widehat{c} : V \rightarrow [m] }{
        \widehat{c} \text{ is a proper coloring with } \widehat{c}(a) = c(a)
        \text{ for all } a \in A}
    \]
    for all $m \ge k$.
\end{thm}

We give an alternative proof of their result and a
combinatorial interpretation for $\chi_{G,c}(-m)$ extending the combinatorial
reciprocity of Stanley~\cite{stanley73} for the ordinary chromatic polynomial.
Recall that an \Emph{orientation} $\sigma$ of $G$ assigns every edge $e$ a
head and a tail.  An orientation is \Emph{acyclic} if there are no directed
cycles. An orientation $\sigma$ is weakly compatible with a given coloring $c:
V \rightarrow [m]$ if $\sigma$ orients an edge $e = uv$ along its color
gradient, that is, from $u$ to $v$
whenever $c(u) < c(v)$. 

\begin{thm} \label{thm:col_recip}
    Let $G=(V,E)$ be a graph and let $c: A\rightarrow [k]$ be a partial
    coloring for $A \subseteq V$. Let $A_1, A_2,\dots, A_k$ be the partition
    of $A$ into color classes induced by $c$. For $m \ge k$ we have that
    $(-1)^{|V\setminus A|}\,\chi_{G,c}(-m)$ is the number of pairs
    $(\widehat{c},\sigma)$ where $\widehat{c} : V \rightarrow [m]$ is a
    coloring extending $c$ and $\sigma$ is a weakly compatible acyclic
    orientation such that there is no directed path with both endpoints in
    $A_i$ for some $i = 1,2,\dots,k$.
\end{thm}

In the case that no two vertices of $A$ get the same color, the result
simplifies.

\begin{cor}
    Let $G=(V,E)$ be a graph and $A \subseteq V$. If $c : A \rightarrow [k]$
    is injective and $m \ge k$, then $|\chi_{G,c}(-m)|$ equals the number of
    pairs $(\widehat{c},\sigma)$ where $\widehat{c}$ is an $m$-coloring
    extending $c$ and $\sigma$ is an acyclic orientation weakly compatible
    with $\widehat{c}$.
\end{cor}

It is also possible to give an interpretation for the evaluations at $-m$ for
$m < k$. Here, we constrain ourselves to one particularly interesting
evaluation.

\begin{cor} \label{cor:special_acyclic}
    Let $G = (V,E)$ be a graph and $c : A \rightarrow [k]$ a partial coloring
    for $A \subseteq V$. Then $|\chi_{G,c}(-1)|$ equals the number of acyclic
    orientations of $G$ for which there is no directed path from $a$ to $b$
    whenever $a, b \in A$ with
    $c(a) \geq c(b)$.     
\end{cor}

Furthermore, choosing $A = \emptyset$, we see that $\chi_{G,c} = \chi_G$ and
the above theorem specializes to the classical reciprocity for chromatic
polynomials.

\begin{cor}[{\cite[Thm.~1.2]{stanley73}}] 
    For a graph $G$, $|\chi_G(-m)|$ equals the number of pairs $(c,\sigma)$
    for which  $c$ is an $m$-coloring and $\sigma$ is a weakly compatible
    acyclic orientation. In particular, $|\chi_G(-1)|$ is the number of
    acyclic orientations of $G$.
\end{cor}

\begin{proof}[Proofs]
    First observe that we may assume that no two vertices of $A$ are assigned
    the same color by $c$. Indeed, assume that  $c(a) = c(b)$ for some $a,b
    \in A$. If $ab$ is an edge of $G$, then no proper coloring can extend $c$
    and $\chi_{G,c} \equiv 0$. Moreover, in any orientation of $G$ there is a
    directed path between $a$ and $b$.  If $ab \not\in E$, let $G_{ab}$
    be obtained from $G$ by identifying $a$ and $b$. Then $c$ descends to a
    partial coloring $c_{ab}$ on $G_{ab}$ and it is easy to see
    that there is a bijective correspondence between extensions of size $m$ of
    $c$ and $c_{ab}$. As for acyclic orientations, note that an acyclic
    orientation of $G$ yields an acyclic orientation of $G_{ab}$  if
    and only if there is no directed path between $a$ and $b$. So, henceforth
    we assume that $c: A \rightarrow [k]$ is injective.

    Let $G^\prime$ be the \Emph{suspension} of $G$, that is, the graph $G$
    with two additional vertices $\hat{0}, \hat{1}$ that are connected to all
    vertices of $G$. For $m \ge k$, let us consider all extensions of $c$ to proper colorings
    $\widehat{c} : V^\prime \rightarrow \{0,1,\dots,m+1\}$ 
    such that $\widehat{c}(\hat{0}) = 0$ and $\widehat{c}(\hat{1}) = m+1$. Every such coloring
    $\widehat{c}$ gives rise to a unique compatible acyclic orientation
    $\sigma$ by directing every edge along its color gradient. By definition,
    $\hat{0}$ is a source and $\hat{1}$ is a sink.  The acyclicity of $\sigma$
    implies that we can define a partially ordered set $P_\sigma$ on
    $V^\prime$ by setting $u \preceq_{P_\sigma} v$ if there is directed path
    from $u$ to $v$.  Extending $A$ to $A^\prime = A \cup \{\hat{0},\hat{1}\}$
    and $c$ to $c^\prime_{m}$ by 
    \[
    c^\prime_m(a) \ =  \
    \begin{cases} 
        0, & \text{if } a = \hat{0},\\
        m+1, & \text{if } a = \hat{1}, \text{ and}\\
        c(a), & \text{otherwise},
    \end{cases}
    \]
    it follows that every proper coloring $\widehat{c}$ of $G^\prime$ that
    extends $c^\prime_m$ and induces $\sigma$ is a strict order
    preserving map $\widehat{c} : P_\sigma \rightarrow \{0,1,\dots,m+1\}$ extending 
    $c^\prime_m$ and vice versa. By Theorem~\ref{thm:reciprocity} 
    \begin{equation}\label{eqn:col_ext}
        \chi_{G,c}(m) \ = \ \sum_\sigma
        (-1)^{|V\setminus A|}\Omega_{P_\sigma,A^\prime}(-c^\prime_m)
    \end{equation}
    where the sum is over all acyclic orientations of $G^\prime$ such that
    for every $a,b \in A^\prime$ there is no directed path from $a$ to $b$
    whenever $c(a) > c(b)$. This shows that $\chi_{G,c}(m)$ is a sum of
    polynomials in $m$ with positive leading coefficients. For $m$ sufficiently
    large, there is an extension of $c$ such that every vertex $V \setminus A$
    gets a color $>k$. For the corresponding poset $P_\sigma$ the summand $\Omega_{P_\sigma,A^\prime}(-c^\prime_m)$ is of
    degree $|V|-|A|$ in $m$ which completes the proof of
    Theorem~\ref{thm:col_ext}.

    Let $A^\prime = \{ \hat{0} = a_0, a_1, \dots, a_{r-1},a_r = \hat{1} \}$ so
    that $i < j$ implies $c^\prime_m(a_i) < c^\prime_m(a_j)$. That is,
    $c^\prime_m$ is a strictly order preserving map for the chain $A^\prime$
    with $c^\prime_m(\hat{0})=0$ and $c^\prime_m(\hat{1})=m+1$.  Hence, we can
    consider the right hand side of \eqref{eqn:col_ext} as a polynomial in $(0
    = c_0 < c_1 < c_2 < \cdots < c_r = m)$.  However, the number of proper
    extensions of $c$ is independent of the actual values of $c : A
    \rightarrow [k]$. Indeed, if $d : A \rightarrow [k]$ is a different
    injective partial coloring, then the permutation $\pi : [k] \rightarrow
    [k]$ that takes $c$ to $d$ extends to a relabeling on every extension of
    $c$ to $d$. It follows that the right hand side of \eqref{eqn:col_ext} is
    a polynomial independent of $c_1,\dots,c_{r-1}$ and
    \[
        (-1)^{|V\setminus A|}\chi_{G,c}(-m) 
        \ = \ \sum_\sigma \Omega_{P_\sigma,A^\prime}(-c^\prime_{-m})
        \ = \ \sum_\sigma \Omega_{P_\sigma,A^\prime}(c^\prime_{m-2} - \chi_A)
    \]
    where $\chi_A : A \rightarrow \{0,1\}$ is the characteristic
    function on $A$.
    Every summand is the number of order preserving maps $P_\sigma
    \rightarrow \{0,1,\dots,m-1\}$ extending $c^\prime_{m-2} - \chi_A$.
    Translating back, this is exactly the number of pairs of (not necessarily
    proper) extensions $\widehat{c}$ of $c^\prime_m$ and a weakly compatible
    acyclic orientation $\sigma$ which yields Theorem~\ref{thm:col_recip}.
    As the right hand side of \eqref{eqn:col_ext} is independent of
    $c_1,\dots,c_{r-1}$ we get that \[
    (-1)^{|V\setminus A|}\chi_{G,c}(-1) 
        \ = \ \sum_\sigma \Omega_{P_\sigma,A^\prime}(-c^\prime_{-1})
        \ = \ \sum_\sigma \Omega_{P_\sigma,A^\prime}(\mathbf{0})
    \]
    Here every summand is one, so the right hand side counts the number of
    acyclic orientations such that for every $a,b \in A$ there is no directed
    path from $a$ to $b$ whenever $c(a) > c(b)$ which proves
    Corollary~\ref{cor:special_acyclic}.
\end{proof}

\begin{example}
Consider the following graph $G$ with $A=\{a,b\}$:

\begin{figure}[h]
\includegraphics[width=1.7cm]{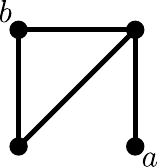}
\end{figure}

Let $c\colon A\rightarrow [k]$ be a coloring. If $c(a)=c(b)$ then for all $m\geq k$ the number of extension of $c$ to a proper $m$-coloring of $G$ is
\[
\chi _{G,c}(m)=(m-1)(m-2)
\]
and $(-1)^2 \chi _{G,c}(-1)=6$ is the number of acyclic orientations of $G$ where there is no directed path between $a$ and $b$:

\begin{figure}[h]
\includegraphics[width=14cm]{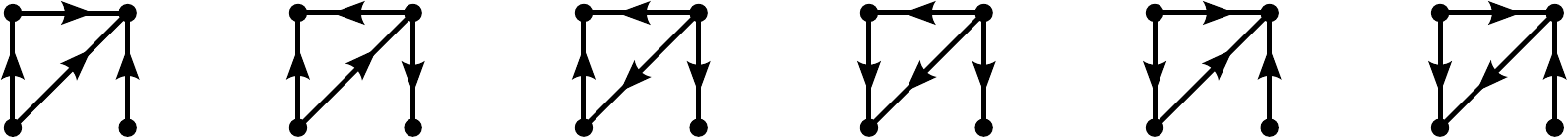}
\end{figure}

If $c(a)> c(b)$ then 
\[
\chi _{G,c}(m)=(m-2)(m-3)+(m-2)=(m-2)^2
\]
and $(-1)^2 \chi _{G,c}(-1)=9$ counts the number of acyclic orientations where
there is no directed path from $a$ to $b$, i.e.\ there are three additional
acyclic orientations:
\begin{figure}[h]
\includegraphics[width=6.0cm]{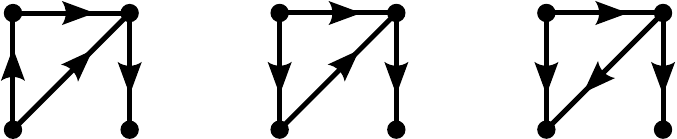}
\end{figure}

The case $c(a) < c(b)$ is clearly analogous.

\end{example}

\bibliographystyle{siam}
\bibliography{MarkedOrderEhrhart}

\begin{thebibliography}{10}

\bibitem{ABS}
{\sc F.~Ardila, T.~Bliem, and D.~Salazar}, {\em {G}elfand--{T}setlin polytopes
  and {F}eigin--{F}ourier--{L}ittelmann--{V}inberg polytopes as marked poset
  polytopes}, J. Combin. Theory Ser. A, 118 (2011), pp.~2454--2462.

\bibitem{BR}
{\sc M.~Beck and S.~Robins}, {\em Computing the continuous discretely:
  {I}nteger-point enumeration in polyhedra}, Undergraduate Texts in
  Mathematics, Springer, New York, 2007.

\bibitem{bliem}
{\sc T.~Bliem}, {\em Chopped and sliced cones and representations of
  {K}ac-{M}oody algebras}, J. Pure Appl. Algebra, 214 (2010), pp.~1152--1164.

\bibitem{DCP}
{\sc C.~De~Concini and C.~Procesi}, {\em Topics in hyperplane arrangements,
  polytopes and box-splines}, Universitext, Springer, New York, 2011.

\bibitem{DLRS}
{\sc J.~A. De~Loera, J.~Rambau, and F.~Santos}, {\em Triangulations:
  {S}tructures for algorithms and applications}, vol.~25 of Algorithms and
  Computation in Mathematics, Springer-Verlag, Berlin, 2010.

\bibitem{fischer06}
{\sc I.~Fischer}, {\em The number of monotone triangles with prescribed bottom
  row}, Adv. in Appl. Math., 37 (2006), pp.~249--267.

\bibitem{FR}
{\sc I.~Fischer and L.~Riegler}, {\em Combinatorial reciprocity for {M}onotone
  {T}riangles}, J. Combin. Theory Ser. A, 120 (2013), pp.~1372--1393.

\bibitem{herzberg}
{\sc A.~M. Herzberg and M.~R. Murty}, {\em Sudoku squares and chromatic
  polynomials}, Notices Amer. Math. Soc., 54 (2007), pp.~708--717.

\bibitem{rota}
{\sc J.~P.~S. Kung, G.-C. Rota, and C.~H. Yan}, {\em Combinatorics: the {R}ota
  way}, Cambridge Mathematical Library, Cambridge University Press, Cambridge,
  2009.

\bibitem{mcmullen73}
{\sc P.~McMullen}, {\em Representations of polytopes and polyhedral sets},
  Geometriae Dedicata, 2 (1973), pp.~83--99.

\bibitem{MRR}
{\sc W.~H. Mills, D.~P. Robbins, and H.~Rumsey, Jr.}, {\em Alternating sign
  matrices and descending plane partitions}, J. Combin. Theory Ser. A, 34
  (1983), pp.~340--359.

\bibitem{stanley70}
{\sc R.~P. Stanley}, {\em A chromatic-like polynomial for ordered sets}, in
  Proc. {S}econd {C}hapel {H}ill {C}onf. on {C}ombinatorial {M}athematics and
  its {A}pplications ({U}niv. {N}orth {C}arolina, {C}hapel {H}ill, {N}.{C}.,
  1970), Univ. North Carolina, Chapel Hill, N.C., 1970, pp.~421--427.

\bibitem{stanley73}
\leavevmode\vrule height 2pt depth -1.6pt width 23pt, {\em Acyclic orientations
  of graphs}, Discrete Math., 5 (1973), pp.~171--178.

\bibitem{stanley81}
\leavevmode\vrule height 2pt depth -1.6pt width 23pt, {\em Two combinatorial
  applications of the {A}leksandrov-{F}enchel inequalities}, J. Combin. Theory
  Ser. A, 31 (1981), pp.~56--65.

\bibitem{stanley86}
\leavevmode\vrule height 2pt depth -1.6pt width 23pt, {\em Two poset
  polytopes}, Discrete Comput. Geom., 1 (1986), pp.~9--23.

\bibitem{ziegler}
{\sc G.~M. Ziegler}, {\em Lectures on {P}olytopes}, vol.~152 of Graduate Texts
  in Mathematics, Springer-Verlag, New York, 1995.

\end{thebibliography}

\end{document}